\documentclass{amsart}
\usepackage[left=2cm, right=2cm]{geometry}
\usepackage[utf8]{inputenc}
\usepackage{amsfonts, amsthm, amssymb, mathtools,etoolbox}
\usepackage{blindtext}
\usepackage[colorlinks=true, urlcolor=blue, linkcolor=blue, citecolor=blue]{hyperref}
\usepackage{tikz,tikz-cd}
\usepackage{array}
\usepackage[style=alphabetic,sorting=nyt]{biblatex}
\renewbibmacro{in:}{}
\tikzset{pullback/.style={minimum size=1.2ex,path picture={
\draw[opacity=1,black,-,#1] (-0.5ex,-0.5ex) -- (0.5ex,-0.5ex) -- (0.5ex,0.5ex);%
}}}
\usepackage{quiver}

\theoremstyle{plain}
\newtheorem{theorem}{Theorem}[section]
\newtheorem{proposition}[theorem]{Proposition}
\newtheorem{lemma}[theorem]{Lemma}
\newtheorem{corollary}[theorem]{Corollary}

\theoremstyle{definition}
\newtheorem{example}[theorem]{Example}
\newtheorem{definition}[theorem]{Definition}
\newtheorem*{definition*}{Definition}

\theoremstyle{remark}
\newtheorem{remark}[theorem]{Remark}

\newcommand{\dq}[1]{``#1"}

\newcommand{\revmemo}[1]{}

\newcommand{\Nor}{\mathrm{N}}

\newcommand{\C}{\mathcal{C}}

\newcommand{\E}{\mathcal{E}}
\newcommand{\F}{\mathcal{F}}
\newcommand{\G}{\mathbb{G}}
\newcommand{\id}{\mathrm{id}}

\newcommand{\ob}{\mathrm{ob}}
\newcommand{\true}{\mathrm{true}}
\newcommand{\Image}{\mathrm{Im}}
\newcommand{\Sub}{\mathrm{Sub}}
\newcommand{\Set}{\mathbf{Set}}
\newcommand{\Group}{\mathbf{Grp}}
\newcommand{\FinSet}{\mathbf{FinSet}}
\newcommand{\PSh}{\mathbf{PSh}}
\newcommand{\Sh}{\mathbf{Sh}}
\newcommand{\sgt}{\{\cdot\}}

\newcommand{\demph}[1]{\textit{#1}}

\newcommand{\yo}{\mathrm{y}}

\newcommand{\mono}{\mathrm{mono}}
\newcommand{\A}{\Sigma}
\newcommand{\MA}{{{\Sigma}^{\ast}}}

\newcommand{\Aset}{\A\text{-}\Set}
\newcommand{\Lan}{\mathcal{L}}

\newcommand{\Pow}{\mathcal{P}}
\newcommand{\of}{\mathrm{o.f.}}

\newcommand{\ofAset}{{\Aset}_{\of}}
\newcommand{\Cont}{\mathbf{Cont}}

\title{Normalization of a subgroup, in a topos, and of a word-congruence}
\author{Ryuya Hora}
\thanks{Graduate School of Mathematical Sciences, University of Tokyo. \url{hora@ms.u-tokyo.ac.jp}}

\begin{document}
\begin{abstract}
This paper provides a new categorical definition of a normalization operator motivated by topos theory and its applications to algebraic language theory.

We first define a normalization operator $\Xi \to \Xi$ in any category that admits a colimit of all monomorphisms $\Xi$, which we call a local state classifier. In the category of group actions for a group $G$, this operator coincides with the usual normalization operator, which takes a subgroup $H\subset G$ and returns its normalizer subgroup $\Nor_G(H)\subset G$.

Using this generalized normalization operator, we prove a topos-theoretic proposition that provides an explicit description of a local state classifier of a hyperconnected quotient of a given topos.
We also briefly explain how these results serve as preparation for a topos-theoretic study of regular languages, congruences of words, and syntactic monoids.
\end{abstract}

\maketitle
\tableofcontents
\section{Introduction}
\subsection{Abstract definitions of a normalizer subgroup}\label{ssec:AbstractDefinitionOfNormalizer}
For a group $G$ and a subgroup $H\subset G$, the \demph{normalizer} of $H$ is defined by
\[
\Nor_{G}(H) \coloneqq \{g\in G \mid g^{-1}Hg=H\}.
\]
In the context of categorical algebra, the normalizer is abstractly described in \cite{gray2014normalizers} as follows: In a category $\C$ with a zero object, prototypically the category of groups $\C = \Group$, a monomorphism $m \colon S \rightarrowtail X$ is said to be normal if it is the kernel of an arrow from $X$. A normalizer of a subobject $m \colon H \rightarrowtail G$ in $\C$ is defined to be the terminal object of the category of factorizations of $m$ as a normal monomorphism followed by a monomorphism \cite[Definition 2.1]{gray2014normalizers}. This is a natural way to generalize the group-theoretic notion of a normalizer $\Nor_{G}(H)$ as the maximal subgroup of $G$ that makes the inclusion $H\subset \Nor_{G}(H)$ normal.

In this paper, we provide another abstract description of normalization — in a way that was (at least to the author) completely unexpected. Instead of considering the category of groups $\Group$,
we consider the category of right $G$-actions, i.e., the presheaf category $\PSh(G)$ on a given group $G$. By equipping the set of all subgroups $\Sub_{\Group}(G)$ with the right conjugate action $H\cdot g \coloneqq g^{-1}Hg$, we can regard the normalization operator
\begin{equation}\label{eq:NormalizationOperatorForAGroup}
    \Nor_G \colon \Sub_{\Group}(G) \to \Sub_{\Group}(G)
\end{equation}
as a morphism in the category $\PSh(G)$.
In this paper, we will describe this normalization operator purely categorically.

Let us overview our abstract construction.
First, we take the colimit of all monomorphisms in the category $\PSh(G)$, which exists nontrivially, and let us write $\Xi$ for it. The colimit cocone is a family of morphisms from all objects in the category, for which we will write $\xi_{X}\colon X \to \Xi$ for every $X \in \ob(\PSh(G))$. Then the \dq{self-referential} component of the colimit cocone $\xi_{\Xi}\colon \Xi \to \Xi$ coincides with the normalization operator $\Nor_G$ (Example \ref{exmp:CaseOfGroup})!


In general, we can define the normalization operator as follows:
\begin{definition*}[Paraphrase of Definition \ref{def:NormalizationOperator}]
For a category $\E$ that admits the colimit of all monomorphisms $\Xi$ with the colimit cocone $\{\xi_{X}\colon X \to \Xi\}_{X\in \ob(\E)}$, 
    the \demph{normalization operator} in $\E$ is the endomorphism $\xi_{\Xi}\colon \Xi \to \Xi.$
\end{definition*}
Such an object $\Xi$ exists in all Grothendieck topoi and, in particular, in all presheaf categories. Therefore, our generalized normalization operator appears in many contexts that are completely different from group theory. As examples, we will see the normalization operator for directed graphs (Example \ref{exmp:ToposOfDirectedGraphTwo}) and free monoid ($=$ words) actions (Section \ref{sec:MotivationgExample}). We also generalize the obvious inequality $H \subset \Nor_G(H)$ for a general context (Proposition \ref{prop:NormalizationLemma}), which plays a central role in the proof of the main theorem.

\subsection{Local state classifier in a hyperconnected quotient topos}\label{ssec:MotivationFromLSC}
A theoretical motivation of this paper comes from topos theory, especially the study of \demph{hyperconnected geometric morphisms}.

Since \cite{johnstone1981factorization} introduced the notion of hyperconnected geometric morphisms, topos theory has heavily utilized it. For example, in the study of topological monoid actions \cite{rogers2023toposes}, hyperconnected geometric morphisms play a central role. As explained in Subsection \ref{ssec:MotivationFromLanguages} and Section \ref{sec:MotivationgExample}, they are also important in the topos-theoretic approach to automata theory. Referring to the terminology `quotient topos' in \cite{lawvere2025open}, we call (an equivalence class of) a hyperconnected geometric morphism from a topos $\E$ a \demph{hyperconnected quotient} of $\E$. A way to enumerate all hyperconnected quotients is first given by \cite{rosenthal1982quotient} using generators of a given Grothendieck topos.

In order to obtain a canonical and simpler classification of hyperconnected quotients,
the colimit of all monomorphisms, denoted by $\Xi$, was introduced under the name \demph{the local state classifier} in the author's paper \cite{hora2024internal} .
The main theorem of the paper \cite{hora2024internal} states that if a topos has a local state classifier $\Xi$, then hyperconnected quotients of $\E$ are in a one-to-one correspondence with internal filters of $\Xi$. 
This result provides a convenient way to classify all the hyperconnected quotients of a broad class of topoi, including all Grothendieck topoi. Therefore, in order to obtain an explicit description of all the hyperconnected quotients, we need a way to calculate the local state classifier of a given topos.

However, explicitly describing the local state classifier $\Xi$ is not always easy. Although the local state classifier of a presheaf topos is explicitly given by $\Xi(c) = \{\text{quotient objects of $\yo(a)$}\}$ \cite[][Example 3.22]{hora2024internal}, it is not easy 
in the case of 
a non-presheaf topos.

The main theorem of the present paper (Theorem \ref{thm:MainTheorem}) provides a new method for describing the local state classifier of a hyperconnected quotient of a known topos. 
As corollaries, we obtain the description of the local state classifier of the topos of continuous actions $\Cont(G)$ for a given topological group $G$ (Corollary \ref{cor:LSCofTopologicalGroups}) and that of the topos of orbit-finite $\A$-sets $\ofAset$ (Corollary \ref{cor:LocalStateClassifierOfOrbitfiniteAset}).
The generalized normalization operator is utilized in its proof.

\subsection{Word combinatorics in algebraic language theory}\label{ssec:MotivationFromLanguages}
The most concrete motivation for this research is a topos-theoretic approach to algebraic language theory.
This paper serves as a theoretical preparation for the forthcoming paper, \textit{Topoi of automata II}. In automata theory, it is crucial to consider \textit{right congruences} on the words $\MA$ for a given alphabet $\A$. The set of all right congruences turns out to be the local state classifier of the topos $\Aset \coloneqq \PSh(\Sigma^*)$, which plays a central role in the ongoing theory of topoi of automata, in particular, in the topos-theoretic counterpart of the Nerode-congruences and syntactic monoids. Here, since $\Aset$ is a presheaf topos, its local state classifier can be described explicitly. 

However, in order to capture finiteness related to algebraic language theory, we need to consider the topoi of topological—in many cases, profinite—monoid actions (see \cite{hora2024topoi}) and their local state classifier.
This is why the main theorem of the present paper is useful, since every topos of topological monoid actions is a hyperconnected quotient of a monoid action topos, as studied in \cite{rogers2023toposes}. In Section \ref{sec:MotivationgExample}, we will briefly observe how our theory of the normalization operator is related to the combinatorics of words.


\subsection*{Acknowledgement}

The author would like to thank his supervisor, Ryu Hasegawa, for his continuous support and suggestions. He is also grateful to Matias Menni for his discussion on the notion of the local state classifier and to the members of the category theory reading group at RIMS.
He was supported by JSPS KAKENHI Grant Number JP24KJ0837 and the FoPM WINGS Program at the University of Tokyo.

The author would like to note that Theorem \ref{thm:MainTheorem} was also independently proven by Professor Peter T. Johnstone. He mentioned this result in his seminar talk on \cite{hora2024internal} at the \href{https://topos.institute/blog/2025-05-08-topox-seminar/}{TopOx seminar} in May 2025. The author is also grateful to him for our discussions on the local state classifier. This research was also supported by the Grothendieck Institute.

\section{Preliminaries on Hyperconnected quotients and local state classifier}\label{sec:Preliminaries}
This section is a $2$-page summary of the paper \cite{hora2024internal}, which defines and studies the notion of a local state classifier.

\subsection{Hyperconnected quotients}
This subsection aims to recall the preliminaries on hyperconnected geometric morphisms. See \cite{johnstone1981factorization} or \cite[][A.4.6]{johnstone2002sketchesv1} for more details.
\begin{definition}[Hyperconnected geometric morphisms]
    A geometric morphism $f\colon \E \to \F$ is said to be \demph{hyperconnected} if it is connected (i.e. $f^{\ast}\colon \F \to \E$ is fully faithful) and its counit $\epsilon_X\colon f^{\ast}f_{\ast}\to \id_{\E}$ is monic.
\end{definition}

In this paper, a \demph{hyperconnected quotiet} of a topos $\E$ means (an equivalence class of) a hyperconnected geometric morphism from $\E$.
Since $f^{\ast}$ is fully faithful for a hyperconnected quotient $f\colon \E \to \F$, we can regard $\F$ as a (replete) full subcategory of $\E$. 
With this identification, we will write `$X\in \ob(\E)$ belongs to $\F$' for `$X\in \ob(\E)$ belongs to the essential image of $f^{\ast}$' in this paper. This does not cause any serious problem, since we will not distinguish between two mutually equivalent hyperconnected quotients.

\subsection{Local state classifier}
In this subsection, we will briefly explain the notion of a local state classifier. For more proofs, informal explanations, and examples, see the original article \cite{hora2024internal}.

\subsubsection{Definition}
\begin{definition}[{\cite[][Definition 3.4]{hora2024internal}}]
    The \demph{local state classifier} of a category $\E$ is the colimit of all monomorphisms of $\E$, if it exists. In other words, it is an object $\Xi$ equipped with a family of morphisms $\{\xi_X \colon X \to \Xi\}_{X\in \ob(\E)}$, such that they form a colimit cocone under the faithful embedding functor $\E_{\mono}\rightarrowtail \E$.
\end{definition}
The definition of a local state classifier is quite transcendental, and even a (small-)cocomplete category might not admit a local state classifier. However, we can prove the following proposition:
\begin{proposition}[{\cite[][Section 3.16]{hora2024internal}}]\label{prop:ExistenceForGrothendieck}
    Every Grothendieck topos $\E$ has a local state classifier.
\end{proposition}

\subsubsection{Inducing full subcategories}\label{sssec:InducedFullSub}
How is a local state classifier related to the classification of hyperconnected quotients? Since $\Xi$ is just an object of $\E$ and a hyperconnected quotient is a (very nice) subcategory of $\E$, they might seem unrelated. 

The answer is, in short, that we can construct a full subcategory of $\E$ from any subobject of $\Xi$.
Let $\E$ be a category with a local state classifier $\Xi$.
For any subobject $\iota_F\colon F \rightarrowtail\Xi$
, we can define a full subcategory $\E_F \hookrightarrow \E$ by
\begin{equation}\label{eq:FullSubCondition}
    X\in \ob(\E_F)
\iff
\begin{tikzcd}
    & F\ar[d, rightarrowtail, "\iota_F"]\\
    X\ar[r,"\xi_X"']\ar[ru, dashed, "\exists"]&\Xi.
\end{tikzcd}
\end{equation}
In other words, we define the full subcategory $\E_F$ of $\E$, specifying objects by
\[
\ob(\E_F) \coloneqq \{X\in \ob(\E)\mid \text{ the morphism $\xi_X$ factors through $F\rightarrowtail \Xi$}\}.
\]
In this paper,
for each object $X\in \ob(\E_F)$, we write $\xi_X^F\colon X \to F$ for the unique lift of $\xi_X$ along $\iota_F$
\[
\begin{tikzcd}
    & F\ar[d, rightarrowtail, "\iota_F"]\\
    X\ar[r,"\xi_X"']\ar[ru, "\xi_X^F"]&\Xi.
\end{tikzcd}
\]


\subsubsection{The order structure}
Although the definition of a local state classifier makes sense for any category, it behaves better in cartesian closed categories. First and foremost, in a cartesian closed category, the local state classifier acquires a canonical semilattice structure reflecting the cartesian structure of $\C$.
\begin{proposition}[{\cite[][Proposition 3.27]{hora2024internal}}]\label{prop:SemilatticeStructure}
    If a cartesian closed category (in particular, an elementary topos) $\E$ admits a local state classifier $\{\xi_X\colon X\to \Xi\}_{X\in \ob (\E)}$, there exists a unique internal $\land$-semilattice structure on $\Xi$ such that the diagram
    \[
    \begin{tikzcd}[column sep =5pt]
        &X_1\times \dots \times X_n \ar[ld, "(\xi_{X_1}) \times \dots \times (\xi_{X_n})"']\ar[rd, "\xi_{(X_1 \times \dots \times X_n)}"]&\\
        \Xi^n\ar[rr,"\land"']&&\Xi
    \end{tikzcd}
    \]
    commutes for any finite sequence of objects $X_1, \dots, X_n \in \ob(\E),\; n\geq 0$.
\end{proposition}

This internal semilattice structure on $\Xi$ induces a (usual) semilattice structure on each homset $\E(X,\Xi)$ for each object $X\in \ob(\E)$. Therefore, each homset $\E(X, \Xi)$ admits a natural partial order defined by $f\leq g \iff f\land g =f$.
A subobject $\iota_F \colon F \rightarrowtail \Xi$ is said to be an \demph{internal filter}, if each subset $\E(X,F) \rightarrowtail \E(X,\Xi)$ is a filter in the usual sense (i.e., upward closed and closed under finite meets $\top, \land$). (In \cite{hora2024internal}, the author adopts a diagrammatic definition of an internal filter so that it makes sense even for locally large categories.)

\subsubsection{The classification theorem}
The paper \cite{hora2024internal} proves that,
if the category $\E$ is an elementary topos
and the subobject $F\rightarrowtail\Xi$ is 
an internal filter,
the induced full subcategory $\E_F$ is also an elementary topos, and the embedding $\E_F \hookrightarrow \E$ admits a right adjoint defining
a hyperconnected geometric morphism $f_F\colon \E \to \E_F$.
The main theorem of \cite{hora2024internal} (Theorem \ref{thm:OldMainTheorem}) states that
this construction $F \mapsto \E_F$ provides a bijective correspondence between the internal filters of $\Xi$ and the hyperconnected quotients of $\E$. 
\begin{theorem}[{\cite[][Theorem 4.1]{hora2024internal}\footnote{In \cite{hora2024internal}, another correspondant, internal semilattice homomorphisms $\Xi \to \Omega$, is given.}}]\label{thm:OldMainTheorem}
    If an elementary topos $\E$ has a local state classifier $\Xi$,
    there exists a bijective correspondence between 
    \begin{itemize}
        \item hyperconnected quotients of the topos $\E$, and
        \item internal filters of the local state classifier $\Xi$.
    \end{itemize}
\end{theorem}

    

    The paper \cite{hora2024internal} also provides the description of the corresponding lex comonad $\G \coloneqq f^{*}f_* \colon \E \to \E$ with its counit $\epsilon \colon \G \to \id_{\E}$.
    \[
    \begin{tikzcd}
        {\;}\ar[rr,phantom, ""'{name=F}]& \E_F \ar[rd,"f^*"]&{\;} \\
        \E \ar[ru,"f_*"]\ar[rr, "\G", ""'{name=U}]\ar[rr, bend right =50, "\id_\E"', ""{name=W}]& & \E
        \ar[to=U, from=F, phantom, "\rotatebox{90}{$\coloneqq$}"]
        \ar[to=W, from=U, Rightarrow, "\epsilon"]
    \end{tikzcd}
    \]
    It states that 
    the monic counit map $\epsilon_X \colon \G X \rightarrowtail X$ for each object $X\in \ob(\E)$ is given by the pullback diagram
    \begin{equation}\label{eq:PullbackDescriptionOfTheCounitAndComonad}
        \begin{tikzcd}
        \G X\ar[r, "\xi^F_{\G X}"]\ar[d, "\epsilon_X", tail]\ar[rd, phantom, "\lrcorner", very near start]& F\ar[d,tail, "\iota_F"]\\
        X\ar[r , "\xi_X"']& \Xi.
    \end{tikzcd}
    \end{equation}

\section{The self-referential aspect of the local state classifier}
In this section, we will state the main theorem without proof in order to motivate the following sections. Then, we will explain the \dq{self-referential} aspect of the local state classifier.

\begin{theorem}\label{thm:MainTheorem}
Let $\E$ be an elementary topos with a local state classifier $\Xi$, and $F \rightarrowtail \Xi$ be an internal filter.
    Then,
    the family of morphisms 
    $\{\xi_{Z}^F \colon Z \to F\}_{Z\in \ob(\E_F)}$ is a local state classifier of the induced hyperconnected quotient topos $\E_{F}$.
\end{theorem}

Notice that the above theorem implicitly states that the filter $F$ belongs to the 
full subcategory
$\E_F$. 
This \dq{self-referential} phenomenon $F \in \ob(\E_F)$ is not trivial. In fact, without the assumption that $F$ is a filter, there are many counter-examples. 
\begin{example}[The topos of graphs: {[Not being a loop] is a loop.} (1/2)] \label{exmp:ToposOfGraphsOne}
    Let us consider the topos of directed graphs $\E \coloneqq \PSh(\rightrightarrows)$. As explained in \cite[][Toy Example 5.3]{hora2024internal}, its local state classifier $\Xi$ looks like
    \[
    \Xi = \left (
    \begin{tikzcd}[scale=3]
    \bullet\ar[loop left,"\text{[Being a loop]}"]\ar[loop right,"\text{[Not being a loop]}"]
    \end{tikzcd}
    \right ).
    \]
    For a directed graph $X = (s,t\colon E \rightrightarrows V)$ in $\E$, the graph morphism $\xi_X$ sends every vertex to the unique vertex of $\Xi$, and sends each edge $e\in E$ to either [Being a loop] or [Not being a loop] detecting whether the edge $e$ is a loop or not.

    Let us consider a subgraph
    \[
    F = \left (
    \begin{tikzcd}[scale=3]
    \bullet\ar[loop right,"\text{[Not being a loop]}"]
    \end{tikzcd}
    \right ),
    \]
    which is not an internal filter. Then the induced full subcategory $\E_F$ consists of graphs whose edges are not loops (i.e., $s(e) \neq t(e)$ for any $e\in E$). Obviously, $F$ itself does not belongs to the subcategory, since \textbf{the edge [Not being a loop] is a loop!} Thus we obtain an example of the situation $F\notin \ob(\E_F)$.
\end{example}






The main subject in the next section, \demph{the normalization operator,} precisely describes such a \dq{self-referential aspect} of the local state classifier (see also Example \ref{exmp:ToposOfDirectedGraphTwo}). By using it,
Corollary \ref{cor:FilterLivesInHQuotient} shows that the condition $F \in \ob(\E_F)$ holds for any upward closed $F$.
\revmemo{cite Kit}

\section{Normalization operator in a category with a local state classifier}
The aim of this section is to define and study what we call the normalization operator of a category.

\subsection{Definition and examples}
\begin{definition}[Normalization operator]\label{def:NormalizationOperator}
    For a category $\E$ that admits a local state classifier $\Xi$, \demph{the normalization operator} 
    \[\xi_{\Xi}\colon \Xi \to \Xi\]
    is the component of the colimit cocone $\{\xi_X \colon X \to\Xi\}_{X\in \ob(\E)}$ at the object $\Xi$.
\end{definition}

\begin{example}[The topos of group actions]\label{exmp:CaseOfGroup}
    This terminology is inspired by the case of group action topos (see \cite[][Example 3.10]{hora2024internal} for details).
In the topos of right $G$-actions $\PSh(G)$ for a group $G$, the local state classifier $\Xi$ is the set of all subgroups equipped with the right conjugate action
\[
H\cdot g \coloneqq g^{-1}Hg \text{ in }\Xi.
\]
Each component of the colimit cocone $\{\xi_X \colon X \to\Xi\}_{X\in \ob(\E)}$
sends each element $x\in X$ of a $G$-set $X$ to its stabilizer subgroup.
    \[
    \xi_X(x) = \{g\in G\mid x\cdot g = x\}
    \]
Therefore, the normalization operator $\xi_{\Xi} \colon \Xi \to \Xi$ sends a subgroup $H\in \Xi$ to its normalizer group $\Nor_{G}(H) \in \Xi$
\[
\xi_{\Xi}\colon H \mapsto \{g\in G \mid g^{-1}Hg= H\} = \Nor_G (H).
\]
\end{example}

\begin{example}[The topos of graphs: {[Not being a loop] is a loop.} (2/2)]\label{exmp:ToposOfDirectedGraphTwo}
    The normalization operator $\xi_\Xi \colon \Xi \to \Xi$ in the topos of directed graphs $\E = \PSh(\rightrightarrows)$ sends both of two loops [Being a loop] and [Not being a loop] in $\Xi$ 
    \[
    \Xi = \left (
    \begin{tikzcd}[scale=3]
    \bullet\ar[loop left,"\text{[Being a loop]}"]\ar[loop right,"\text{[Not being a loop]}"]
    \end{tikzcd}
    \right )
    \]
    to the edge [Being a loop].
    In particular, we have $\xi_{\Xi}(\text{[Not being a loop]}) = \text{[Being a loop]}$, which captures the self-referential statement \dq{the edge [Not being a loop] is a loop.}
\end{example}

\begin{example}[Localic topos: trivial case]\label{exmp:localic}
    A Grothendieck topos $\E$ is localic if and only if its local state classifier $\Xi$ is a terminal object \cite[5.5. Corollary]{hora2024internal}. In such a localic case, including the sheaf topos $\E=\Sh(X)$ over a topological space $X$, the normalization operator $\xi_{\Xi}\colon \Xi \to \Xi$ trivially coincides with the identity morphism $\id_{\Xi}$.
\end{example}

\begin{remark}[Topoi with $\xi_{\Xi}=\id_{\Xi}$]\label{rmk:ToposWithIdentityOperator}
    The author does not know any topos-theiretic condition that the  normalization operator $\xi_{\Xi}$ coincides with the identity. In the group action topos $\PSh(G)$, the normalization operator cannot be the identity unless $G$ is trivial since $\xi_{\Xi}(\{e\})=G$.
While Example \ref{exmp:localic} implies that such a class of topoi includes all localic topoi, Example \ref{exmp:ToposOfIdempotents} shows that it is properly broader than localic topoi.
\end{remark}

\begin{example}[The topos of idempotent functions: $\xi_{\Xi}=\id_{\Xi}$]\label{exmp:ToposOfIdempotents}
Let us consider the topos of idempotent functions
    $\E\coloneqq \PSh(\langle x\mid x^2=x\rangle)$, which is equivalent to the topos of actions of the monoid $(\mathbb{F}_2,1,\times)$. An object of this topos is a pair $(X, \sigma\colon X \to X)$ of a set $X$ and an idempotent endofunction $\sigma^2=\sigma$.

    The local state classifier $\Xi$ of this topos is the $2$-element set
    \[
    \Xi = \{\text{[fixed]}, \text{[not fixed]}\}
    \]
    equipped with the idempotent morphism
    \[
    \begin{tikzcd}[row sep = 10pt]
        \text{[not fixed]} \ar[r, mapsto, "\sigma"]&\text{[fixed]}\\
        \text{[fixed]}\ar[r, mapsto, "\sigma"]&\text{[fixed]}.
    \end{tikzcd}
    \]
    Each component $\xi_{(X, \sigma)}\colon X \to \Xi$ sends a fixed point $\sigma(x)=x$ to $\text{[fixed]}$ and a non-fixed point $\sigma(x)\neq x$ to $\text{[not fixed]}$. Therefore, $\xi_{\Xi}$ coincides with the identity function
    \[
    \xi_{\Xi}=\id_{\Xi}
    \]
    since $\text{[fixed]}$ is fixed and $\text{[not fixed]}$ is not fixed.
\end{example}

\begin{remark}[Topoi with $\xi_{\Xi}=\top$]\label{exmp:ToposWithTopOperator}
    As the hom set $\E(\Xi,\Xi)$ is a semilattice, we can consider the equation $\xi_{\Xi}=\top$, or equivalently, the condition that the normalization operator $\xi_{\Xi}$ factors through the map $\xi_{1}\colon 1 \rightarrowtail \Xi$. This condition holds for any localic topoi, and in fact, a topos is localic if and only if the two equations $\xi_{\Xi}=\id_{\Xi}$ and $\xi_{\Xi}=\top$ hold by the obvious reason (see Example \ref{exmp:localic} and Remark \ref{rmk:ToposWithIdentityOperator}). For a group $G$, the topos $\PSh(G)$ satisfies the condition $\xi_{\Xi}=\top$ if and only if the group $G$ is a \textit{Dedekind group}, i.e., every subgroup of $G$ is normal. For example, the topos $\PSh(Q_8)$ satisfies the equation $\xi_{\Xi}=\top$.
\end{remark}

While the examples we have seen tend to be idempotent, the normalization operator is usually not idempotent at all. 
\begin{example}[The topos of species: Non-idempotent normalization operator]\label{exmp:DihedralGroup}
For the $4$th Dihedral group $D_4 \coloneqq \langle \sigma, \tau\mid \sigma^4=1, \tau^2=1, \tau\sigma = \sigma^3\tau\rangle$, the normalization operator $\xi_{\Xi}$ in $\PSh(D_4)$, which coincides with the group-theoretic one, is visualized in the following diagram.
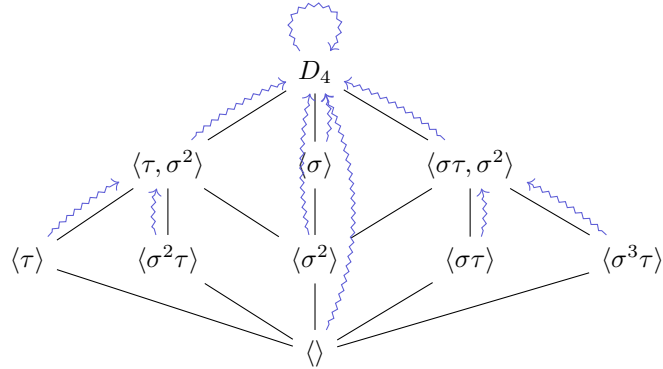
\[\begin{tikzcd}
	&& {D_4} \\
	& {\langle\tau, \sigma^2\rangle} & {\langle \sigma\rangle} & {\langle\sigma\tau, \sigma^2\rangle} \\
	{\langle\tau\rangle} & {\langle\sigma^2\tau\rangle} & {\langle\sigma^2\rangle} & {\langle\sigma\tau\rangle} & {\langle\sigma^3 \tau\rangle} \\
	&& {\langle\rangle}
	\arrow[color={rgb,255:red,92;green,92;blue,214}, squiggly, from=1-3, to=1-3, loop, in=55, out=125, distance=10mm]
	\arrow[no head, from=2-2, to=1-3]
	\arrow[color={rgb,255:red,92;green,92;blue,214}, curve={height=-6pt}, squiggly, from=2-2, to=1-3]
	\arrow[no head, from=2-3, to=1-3]
	\arrow[color={rgb,255:red,92;green,92;blue,214}, curve={height=6pt}, squiggly, from=2-3, to=1-3]
	\arrow[no head, from=2-4, to=1-3]
	\arrow[color={rgb,255:red,92;green,92;blue,214}, curve={height=6pt}, squiggly, from=2-4, to=1-3]
	\arrow[no head, from=3-1, to=2-2]
	\arrow[color={rgb,255:red,92;green,92;blue,214}, curve={height=-6pt}, squiggly, from=3-1, to=2-2]
	\arrow[no head, from=3-2, to=2-2]
	\arrow[color={rgb,255:red,92;green,92;blue,214}, curve={height=-6pt}, squiggly, from=3-2, to=2-2]
	\arrow[color={rgb,255:red,92;green,92;blue,214}, curve={height=-6pt}, squiggly, from=3-3, to=1-3]
	\arrow[no head, from=3-3, to=2-2]
	\arrow[no head, from=3-3, to=2-3]
	\arrow[no head, from=3-3, to=2-4]
	\arrow[no head, from=3-4, to=2-4]
	\arrow[color={rgb,255:red,92;green,92;blue,214}, curve={height=6pt}, squiggly, from=3-4, to=2-4]
	\arrow[no head, from=3-5, to=2-4]
	\arrow[color={rgb,255:red,92;green,92;blue,214}, curve={height=6pt}, squiggly, from=3-5, to=2-4]
	\arrow[color={rgb,255:red,92;green,92;blue,214}, curve={height=18pt}, squiggly, from=4-3, to=1-3]
	\arrow[no head, from=4-3, to=3-1]
	\arrow[no head, from=4-3, to=3-2]
	\arrow[no head, from=4-3, to=3-3]
	\arrow[no head, from=4-3, to=3-4]
	\arrow[no head, from=4-3, to=3-5]
\end{tikzcd}
\]
Therefore, $\xi_{\Xi}$ is not idempotent nor order-preserving. The same argument works for $\PSh(S_4)$ and hence for the topos of species $\PSh(\FinSet_{\mathrm{bij}})$ (see \cite{joyal1981theorie} and \cite[Example. 3.14]{hora2024internal}). This implies that the normalization operator in the topos of species $\PSh(\FinSet_{\mathrm{bij}})$ is not idempotent.
\end{example}

\subsection{The normalization lemma}
At first glance, the normalization operator has nothing to do with the order structure ($=$ the semilattice structure) of $\Xi$. 
In fact, it does not preserve the order structure in general (see Example \ref{exmp:DihedralGroup}).

However, there is an obvious inclusion relation 
\[H \subset \Nor_G(H)\] for the normalizer of a subgroup $H \subset G$, which can be categorically generalized.
\begin{proposition}[Normalization lemma]
\label{prop:NormalizationLemma}
In a cartesian closed category $\E$ with a local state classifier $\Xi$, the morphism $\xi_{\Xi}\colon \Xi\to \Xi$ is equal to or larger than $\id_{\Xi}$
\[
\begin{tikzcd}[column sep=50pt]
    \Xi\ar[r, bend left, ""'{name=A}, "\id_\Xi"]\ar[r, bend right, "\xi_\Xi"', ""{name=B}] \ar[from=A, to=B, phantom, "\rotatebox{90}{$\geq$}"] &\Xi
\end{tikzcd}
\]
with respect to the $\land$-semilattice structure on $\E(\Xi, \Xi)$.
\end{proposition}
\begin{proof}
    To prove $\id_{\Xi}\leq \xi_{\Xi}$, we need to prove that the composite of
    \[
    \begin{tikzcd}[column sep = 50pt]
        \Xi\ar[r,"{\langle\id_{\Xi}, \xi_{\Xi}\rangle}"]&
        \Xi\times \Xi\ar[r,"\land"]&\Xi
    \end{tikzcd}        
    \]
    is the identity.
    Since $\Xi$ is a colimit, it suffices to prove the commutativity of the following diagram for each object $X\in \ob(\E)$.
     \[
    \begin{tikzcd}[column sep = 50pt]
     X\ar[d,"\xi_X"']\ar[rr,bend left, "\xi_X"]&&\Xi\ar[d,equal]\\
        \Xi\ar[r,"{\langle\id_{\Xi}, \xi_{\Xi}\rangle}"]&
        \Xi\times \Xi\ar[r,"\land"]&\Xi
    \end{tikzcd}        
    \]
    By the definition of $\land$ operation and the fact that ${\langle\id_X, \xi_X\rangle}$ is a (split) monomorphism, we have the next commutative diagram.
    \[
    \begin{tikzcd}[column sep = 50pt]
     X\ar[d,"\xi_X"']\ar[rr,bend left, "\xi_X"]\ar[r, "{\langle\id_X, \xi_X\rangle}", tail]&X\times \Xi\ar[d,"{\xi_X \times \xi_{\Xi}}"]\ar[r,"\xi_{X\times \Xi}"]&\Xi\ar[d,equal]\\
        \Xi\ar[r,"{\langle\id_{\Xi}, \xi_{\Xi}\rangle}"]&
        \Xi\times \Xi\ar[r,"\land"]&\Xi
    \end{tikzcd}        
    \]
    This completes the proof.
\end{proof}

\begin{corollary}
\label{cor:FilterLivesInHQuotient}
    For any cartesian closed category $\E$ with a local state classifier $\Xi$ and any internal filter (or, more generally, any upward closed subobject) $F\rightarrowtail \Xi$, we have 
    \[
    F \in \ob(\E_F),
    \]
    i.e., 
     $F$ belongs to the induced full subcategory $\E_F \hookrightarrow \E$.
\end{corollary}
\begin{proof}

Due to the equivalence (\ref{eq:FullSubCondition}), it suffices to prove that $\xi_F \colon F \to \Xi$ lifts along $\iota_{F}\colon F \rightarrowtail \Xi$.
\[
F\in \ob(\E_F)
\iff
\begin{tikzcd}
    & F\ar[d, rightarrowtail, "\iota_{F}"]\\
    F\ar[r,"\xi_F"']\ar[ru, dashed, "\exists"]&\Xi
\end{tikzcd}
\]
Since $\iota_{F}$ trivially lifts along itself and the internal filter $F$ is upward closed, it is enough to prove the inequality $\iota_{F} \leq \xi_F$. This follows from the following diagram and the inequality of Proposition \ref{prop:NormalizationLemma}.
\[
\begin{tikzcd}[column sep=50pt]
    F \ar[r,"\iota_{F}", rightarrowtail]\ar[rr, bend right=50, "\xi_F"',""{name=C}]&\Xi\ar["\rotatebox{90}{$=$}", to={C}, phantom]\ar[r, bend left, ""'{name=A}, "\id_\Xi"]\ar[r, bend right, "\xi_\Xi"', ""{name=B}] \ar[from=A, to=B, phantom, "\rotatebox{90}{$\geq$}"] &\Xi
\end{tikzcd}
\]
    
\end{proof}

\section{Local state classifier in a hyperconnected quotient}
The goal of this section is to prove Theorem \ref{thm:MainTheorem}. 
Throughout this section, we fix the following data.
\begin{itemize}
    \item $\E$ is an elementary topos with a local state classifier $\Xi$.
    \item $F$ is an internal filter of $\Xi$.
    \item $\E_F$ is the corresponding hyperconnected quotient of $\E$.
    \item $\G$ is the corresponding (lex idempotent) comonad on $\E$ with the monic counit $\{\epsilon_X \colon \G X \rightarrowtail X\}_{X\in \ob(\E)}$
\end{itemize}
Due to Theorem \ref{thm:OldMainTheorem} and Proposition \ref{prop:ExistenceForGrothendieck}, this situation subsumes all hyperconnected geometric morphisms between Grothendieck topoi.

Due to Corollary \ref{cor:FilterLivesInHQuotient}, we know that the filter $F$ and all the components of the family $\{\xi_{Z}^F \colon Z \to F\}_{Z\in \ob(\E_F)}$ belong to the full subcategory $\E_F$. As the next lemma shows, it is not hard to prove that it is a cocone.

\begin{lemma}[Being a cocone]\label{lem:BeingCocone}
     The family $\{\xi_{Z}^F \colon Z \to F\}_{Z\in \ob(\E_F)}$ is a cocone under the functor ${(\E_F)}_{\mono} \to \E_F$.
\end{lemma}
\begin{proof}
    Let $m\colon Z\rightarrowtail Z'$ be an arbitrary monomorphism in the category $\E_F$. Since the embedding $\E_F \hookrightarrow \E$ preserves finite limits, $m$ remains monic in the ambient topos $\E$. Therefore, we have the commutativity of the outer perimeter of the following diagram
    \[
    \begin{tikzcd}
        Z\ar[rr,"m", rightarrowtail]\ar[rd, "\xi^F_{Z}"']\ar[rdd, "\xi_Z"', bend right]&&Z'\ar[ld, "\xi^F_{Z'}"] \ar[ldd, "\xi_Z'", bend left]\\
        &F\ar[d,rightarrowtail, "\iota_F"]&\\
        &\Xi.&
    \end{tikzcd}
    \]
    Since $\iota_F$ is monic, this implies the commutativity of the inner triangle $ \xi_{Z'}^F \circ m = \xi^F_{Z}$, which completes the proof.
\end{proof}

In the rest of this section, we will prove the universality of the family $\{\xi_{Z}^F \colon Z \to F\}_{Z\in \ob(\E_F)}$ as a colimit of the functor $(\E_F)_{\mono} \to \E_F$. What we can use is the fact that the cocone $\{\xi_X \colon X \to \Xi\}_{X\in \ob(\E)}$ is a (large) colimit cocone of $\E_{\mono}\to \E$. So we will convert the situations in $\E_F$ to the larger category $\E$ and reduce the required universality of $F \in \ob(\E_F)$ to that of $\Xi \in \ob(\E)$.

First, we will prove the uniqueness part of the universality of $F$. Recall that a (possibly large) family of morphisms $\{f_\lambda \colon X_\lambda \to Y\}_{\lambda \in \Lambda}$ is said to be \demph{jointly epimorphic} if, for any parallel morphisms $g,h\colon Y \rightrightarrows Z$, the implication
$
(\forall \lambda \in \Lambda,\;  g\circ f_\lambda = h\circ f_\lambda) \implies (g=h)
$ holds.
In the following proof, we will not assume that the topos $\E$ is a Grothendieck topos. So we cannot use the complete lattice structure of the subobject lattice of $\Xi$. Instead of that, we use the Heyting algebra structure, which makes sense in an arbitrary elementary topos.

\begin{lemma}[Universality (1/2): Uniqueness]
\label{lem:JointlyEpimorphic}
    The cocone $\{\xi_{Z}^F \colon Z \to F\}_{Z\in \ob(\E_F)}$ is jointly epimorphic (in $\E$, and hence in $\E_F$).
\end{lemma}
\begin{proof}
    Take an arbitrary subobject $S\rightarrowtail F$ such that every arrow in the family $\{\xi^F_{Z} \colon Z \to F\}_{Z\in \ob(\E_F)}$ lifts along $S\rightarrowtail F$. Since $\E$ is a topos, it suffices to prove $S=F$ (see Lemma \ref{lem:JointlyEpimorphicFamilyAndSubobject}).
    The lifting assumption on $S$ is rephrased as
    the inequality 
\[
\Image(\xi^F_Z) \leq S \text{ in } \Sub_{\E}(\Xi)
\]
for every object $Z\in \ob(\E_F)$.
We also have the equality
    \[
    \Image(\xi_X)\land F = \Image(\xi^F_{\G X}) \text{ in } \Sub_{\E}(\Xi)
    \]
    for every object $X \in \ob(\E)$ due to the Beck-Chevalley condition 
    for the pullback square (\ref{eq:PullbackDescriptionOfTheCounitAndComonad})
    \[
    \begin{tikzcd}
        \G X\ar[r, "\xi_{\G X}^F"]\ar[d, "\epsilon_X", tail] \ar[rd, phantom, "\lrcorner", very near start]& F\ar[d,tail, "\iota_F"]\\
        X\ar[r , "\xi_X"']& \Xi.
    \end{tikzcd}
    \]
    Combining the above two (in)equalities in $\Sub_{\E}(\Xi)$, 
    we obtain 
    \[
    \Image(\xi_X)\land F  \leq S \text{ in } \Sub_{\E}(\Xi)
    \]
    for each object $X\in \ob(\E)$.
    Since the poset $\Sub_{\E}(\Xi)$ is a Heyting algebra,
    this is equivalent to 
    \[
    \Image(\xi_X) \leq (F \mathbin{\rightarrow}  S) \text{ in } \Sub_{\E}(\Xi).
    \] 
    This means that the colimit cocone $\{\xi_X\colon X \to \Xi\}_{X\in \ob(\E)}$ factors through the subobject $(F\to S) \in \Sub_{\E}(\Xi)$.
    Since the colimit cocone $\{\xi_X \colon X \to \Xi\}_{X \in \ob(\E)}$ is jointly epimorphic, Lemma \ref{lem:JointlyEpimorphicFamilyAndSubobject} implies
    \[
    \Xi = \top= (F  \mathbin{\rightarrow} S) \text{ in } \Sub_{\E}(\Xi),
    \]  
    i.e., $F\leq S$. Since $S\leq F$ holds by definition, this completes the proof. 
\end{proof}

Lastly, we need to prove the existence part of the universality of $F$, using the universality of $\Xi$. This is the tricky part, since in order to use the universality of $\Xi$, we need to construct a cocone under the functor $\E_{\mono} \to \E$ from a given cocone $\{\phi_Z\colon Z\to L\}_{Z \in \ob(\E_F)}$ under the smaller functor $(\E_F)_{\mono} \to \E_F$. In other words, we need to extend the index class of a cocone from $\ob(\E_F)$ to $\ob(\E)$.

The first idea for the extension problem is to use the counit $\epsilon_X \colon \G X \rightarrowtail X$. Since $\G X$ is an object of $\E_F$, we can canonically associate an object $\G X$ of $\E_F$ with each object $X$ of $\E$. But here is another problem. Although we want to construct a family of morphisms \textbf{from} all objects $X$ in $\E$, what the counit provides are morphisms \textbf{to} objects of $\E$. 
\[
\begin{tikzcd}
    \G X \ar[r,"\phi_{\G X}"]\ar[d, rightarrowtail, "\epsilon_X"]& L\\
    X\ar[ru, "?"', dashed] &
\end{tikzcd}
\]
To invert the direction of the arrow, we need to consider the next idea, namely using the powerset object. 
Recall that any morphism $f\colon X \to Y$ in a topos induces three morphisms between the powerset objects
\[
\begin{tikzcd}[column sep = 70pt]
    PX \ar[r,"\exists_f ", bend left] \ar[r, "\forall_f"', bend right ] & PY. \ar[l, "f^{-1}"']
\end{tikzcd}
\]
Combining these ideas, we obtain a cocone under $\E_{\mono} \to \E$ as follows
\[
\begin{tikzcd}
    \G X \ar[r,"\phi_{\G X}"]\ar[d, rightarrowtail, "\epsilon_X"]& L\\
    X\ar[ru, "?"', dashed] &
\end{tikzcd}
\xrightarrow{\text{taking power}}
\begin{tikzcd}
    P\G X \ar[r,"\exists_{\phi_{\G X}}"]& PL\\
    PX \ar[u,  "{\epsilon_X}^{-1}", twoheadrightarrow] &\\
    X\ar[u, "\sgt_X", rightarrowtail] \ar[ruu, "!"', dashed]
\end{tikzcd}
\]
Here, we use the morphism $\exists_{\phi_{\G X}}$, not $\forall_{\phi_{\G X}}$, because of its nicer properties, such as Lemma \ref{lem:sgtNaturality}.
In what follows, the notation $PX$ denotes the powerset object in the topos $\E$, not in the topos $\E_F$.
\begin{lemma}
\label{lem:ExtensionLemma}
    Let 
    $\{\phi_Z\colon Z\to L\}_{Z \in \ob(\E_F)}$ be a cocone under the diagram $(\E_F)_{\mono}\to \E_F$. Then, the family of morphisms
    \[
    \begin{tikzcd}
        \{\psi_X\colon X\ar[r,"\sgt_X"] &PX \ar[r,"{\epsilon_X}^{-1}"]&P\G X\ar[r,"\exists_{\phi_{\G X}}"] & PL\}_{X\in \ob(\E)}
    \end{tikzcd}
    \]
    defines a cocone under the diagram $\E_{\mono} \to \E$.
    Furthermore, the cocone $\psi$ is an extension of $\phi$ in the sense that the diagram
    \begin{equation}\label{eq:PsiExtendsPhi}
        \begin{tikzcd}
        \G X \ar[r, "\phi_{\G X}"]\ar[d, tail, "\epsilon_X"]& L\ar[d, "\sgt_L", tail]\\
        X\ar[r, "\psi_X"]& PL
    \end{tikzcd}
    \end{equation}
    commutes for every object $X \in \ob(\E)$.
\end{lemma}
\begin{proof}
For any monomorphism $m\colon X\rightarrowtail Y$ in $\E$, we have a commutative diagram
    \[
    \begin{tikzcd}[row sep = 10pt]
        X\ar[r,"\sgt_X"] \ar[dd, "m", tail]&PX \ar[r,"{\epsilon_X}^{-1}"]\ar[dd, "\exists_m"]&P\G X\ar[rd,"\exists_{\phi_{\G X}}"]\ar[dd,"\exists_{\G m}"]&\\    &&&PL\\
        Y\ar[r,"\sgt_X"] &PY \ar[r,"{\epsilon_Y}^{-1}"]&P\G Y\ar[ru,"\exists_{\phi_{\G Y}}"'] &
    \end{tikzcd}
    \]
    since the left square commutes by Lemma \ref{lem:sgtNaturality}, the right triangle commutes by the assumption of $\phi$ being a cocone under $(\E_F)_{\mono} \to \E_{F}$, and the middle square commutes by the Beck-Chevalley condition for the pullback square
    \[
    \begin{tikzcd}
        \G X \ar[r, "\G m", tail]\ar[d,"\epsilon_X", tail]\ar[rd, phantom, "\lrcorner", very near start]&\G Y\ar[d,"\epsilon_Y", tail]\\
        X\ar[r,"m", tail] & Y.
    \end{tikzcd}
    \]
    This proves that $\psi$ is a cocone under $\E_{\mono}\to \E$.
    
     To prove the latter part, which states the diagram
    \[
    \begin{tikzcd}
        \G X \ar[rrr, "\phi_{\G X}"]\ar[d, tail, "\epsilon_X"]&&& L\ar[d, "\sgt_L", tail]\\
        X\ar[r,"\sgt_X"'] &PX \ar[r,"{\epsilon_X}^{-1}"']&P\G X\ar[r,"\exists_{\phi_{\G X}}"'] & PL
    \end{tikzcd}
    \]
    is commutative, it suffices to check the commutativity of
    \[
    \begin{tikzcd}
        \G X \ar[rrr, "\phi_{\G X}"]\ar[rrd, "\sgt_{\G X}", tail]\ar[d, tail, "\epsilon_X"]&&& L\ar[d, "\sgt_L", tail]\\
        X\ar[r,"\sgt_X"'] &PX \ar[r,"{\epsilon_X}^{-1}"']&P\G X\ar[r,"\exists_{\phi_{\G X}}"'] & PL.
    \end{tikzcd}
    \]
    Lemma \ref{lem:sgtNaturality} and Lemma \ref{lem:sgtExtNaturality} complete the proof.
\end{proof}

Now we are ready to prove Theorem \ref{thm:MainTheorem}.
\begin{proof}[Proof of Theorem \ref{thm:MainTheorem}]
    In (Corollary \ref{cor:FilterLivesInHQuotient} and) Lemma \ref{lem:BeingCocone}, we have already observed that the family of morphisms 
    $\{\xi_{Z}^F \colon Z \to F\}_{Z\in \ob(\E_F)}$ is a cocone under the functor $(\E_{F})_{\mono} \to \E_F$. 
    It suffices to prove that this family has the universality as a colimit of all monomorphisms in $\E_F$. Since Lemma \ref{lem:JointlyEpimorphic} ensures the uniqueness part of the desired universality, we will prove the existence part.

    Take an arbitrary cocone $\{\phi_Z \colon Z\to L\}_{Z\in \ob(\E_F)}$ under the functor $(\E_{F})_{\mono} \to \E_F$. Let $\{\psi_X\colon X \to PL\}_{X\in \ob(\E)}$ be the cocone under the functor $\E_{\mono} \to \E$ given in Lemma \ref{lem:ExtensionLemma}. By the universality of the local state classifier $\Xi$, we have a unique morphism $\gamma \colon \Xi \to PL$ such that
    \[
    \begin{tikzcd}
    [row sep = 50pt]
        &X\ar[ld,"\xi_X"']\ar[rd,"\psi_X"]&\\
        \Xi\ar[rr,"\gamma"]&&PL
    \end{tikzcd}
    \]
    commutes for every $X\in \ob(\E)$. Due to the diagrams (\ref{eq:PullbackDescriptionOfTheCounitAndComonad}) and (\ref{eq:PsiExtendsPhi}), 
    the following diagram is also commutative.
    \begin{equation}\label{eq:block}
    \begin{tikzcd}
    [row sep = 50pt]
        &\G X\ar[ld, "\xi_{\G X}^F"']\ar[d,"\epsilon_X", tail]\ar[rd, "\phi_{\G X}"]&\\
       F\ar[d,tail, "\iota_F"'] &X\ar[ld,"\xi_X"']\ar[rd,"\psi_X"]&L\ar[d,"\sgt_L", tail]\\
        \Xi\ar[rr,"\gamma"]&&PL
    \end{tikzcd}
    \end{equation}
    Theorefore, it is sufficient to prove that $\gamma \circ \iota_F\colon F\rightarrowtail \Xi \to PL$ lifts along $\sgt_L$
    \[
    \begin{tikzcd}
    [row sep = 50pt]
        &\G X\ar[ld, "\xi_{\G X}^F"']\ar[rd, "\phi_{\G X}"]&\\
       F\ar[d,tail, "\iota_F"'] \ar[rr, dashed, "?"]&&L\ar[d,"\sgt_L", tail]\\
        \Xi\ar[rr,"\gamma"]&&PL,
    \end{tikzcd}
    \]
    since $\sgt_L$ is monic and every object in $\E_{F}$ is isomorphic to an object of the form of $\G X$.
    Take the characteristic morphism of $\sgt_{L}\colon L \rightarrowtail PL$ as
    \[
    \begin{tikzcd}
    [row sep = 50pt]
        &\G X\ar[ld, "\xi_{\G X}^F"']\ar[rd, "\phi_{\G X}"]&&\\
       F\ar[d,tail, "\iota_F"'] \ar[rr, dashed, "?"]&&L\ar[d,"\sgt_L", tail]\ar[rr, "!"]\ar[rrd, phantom, "\lrcorner", very near start]&&1\ar[d,"\true", tail]\\
        \Xi\ar[rr,"\gamma"]&&PL\ar[rr,"\chi_L"]&&\Omega.
    \end{tikzcd}
    \]
    By the universality of the pullback, it suffices to
    prove that the composite $F\rightarrowtail \Xi \to PL \to \Omega$ coincides with the true morphism $\true_{F}\colon F \to \Omega$. 
    The joint surjectivity of $\{\xi_{\G X}^F\colon\G X \to F \}_{X\in \ob(\E)}$, which is an immediate corollary of Lemma \ref{lem:JointlyEpimorphic}, reduces it to proving that the composition $\G X \to F \to \Xi \to PL \to \Omega$ coincides with $\true_{\G X}$ for every $X \in \ob(\E)$. This follows from the commutativity of the perimeter of
    \[
    \begin{tikzcd}
    [row sep = 50pt]
        &\G X\ar[ld, "\xi_{\G X}^F"']\ar[rd, "\phi_{\G X}"]\ar[rrrd, bend left, "!"]&&\\
       F\ar[d,tail, "\iota_F"'] 
       &(\ref{eq:block})&L\ar[d,"\sgt_L", tail]\ar[rr, "!"]\ar[rrd, phantom, "\lrcorner", very near start]&&1\ar[d,"\true", tail]\\
        \Xi\ar[rr,"\gamma"]&&PL\ar[rr,"\chi_L"]&&\Omega.
    \end{tikzcd}
    \]
    This completes the proof. 
\end{proof}

\begin{remark}
    This proof can be shortened if the topos $\E$ is a Grothendieck topos. In fact, in order to ensure the existence of the lift, which is equivalent to the inequality $\exists_{\gamma}(F)\le L \text{ in }\Sub_{\E}(PL)$, it suffices to observe
    \[\exists_{\gamma}(F)=\exists_{\gamma}\left(\bigvee_{X}\Image(\xi^F_{\G X})\right)=\bigvee_{X}\Image(\gamma\circ \iota_F\circ\xi^F_{\G X})=\bigvee_{X}\Image(\sgt_{L}\circ \phi_{\G X})=\exists_{\sgt_{L}}\left(\bigvee_{X}\Image(\phi_{\G X})\right)\le L.
    \]
\end{remark}

Let us provide an example of our main theorem.
\begin{corollary}[Local state classifier of a continuous group action topos]\label{cor:LSCofTopologicalGroups}
    For a topological group $G$, the local state classifier $\Xi$ of the Grothendieck topos of continuous $G$-actions $\Cont(G)$ is given by
    \[
    \Xi=\{\text{open subgroups of }G\},
    \]
    equipped with the right conjugate action
    \[
    H\cdot g \coloneqq g^{-1}Hg.
    \]
\end{corollary}
\begin{proof}
    The Grothendieck topos of continuous $G$-actions, which is denoted by $\Cont(G)$, is a hyperconnected quotient of the presheaf topos $\PSh(G^{\delta})$ on the underlying (discrete) group $G^{\delta}$
    \[
    h\colon \PSh(G^{\delta}) \twoheadrightarrow \Cont(G).
    \]
    The local state classifier of the presheaf topos $\PSh(G^{\delta})$ is the set of all subgroups, and each component $\xi_X$ of the colimit cocone sends an element to its stabilizer. A $G^{\delta}$-set belongs to $\Cont(G)$ if and only if the stabilizer subgroup for any element is open. Therefore, the internal filter $F$ that corresponds to the hyperconnected quotient $\Cont(G)$ is the set of all open subgroups of $G$. (One can easily verify that this is, in fact, an internal filter.)
    Therefore, we complete the proof by applying Theorem \ref{thm:MainTheorem}. 
\end{proof}
Notice that Proposition \ref{prop:NormalizationLemma} and Corollary \ref{cor:FilterLivesInHQuotient} appear here to ensure that the topological group $G$ continuously acts on $\Xi=\{\text{open subgroups of }G\}$ as the normalizer group of an open subgroup is open.

\section{Motivating example: the topos of word actions}\label{sec:MotivationgExample}
This section briefly introduces the initial part of the motivating example, namely the topos of word actions $\Aset$. 
The following contents will be included in the forthcoming paper \textit{Topoi of automata II}, but the author believes that including a nontrivial concrete example will help clarify the motivation of the present paper. Therefore, only the part related to the normalization operator will be outlined below. For the details of the automata-theoretic notions, see \cite{pin2022mathematical}.

\subsection{Right congruences form the local state classifier of the topos \texorpdfstring{$\Aset$}{Aset}}
For a set $\A$, which we call \demph{alphabet}, we consider its free monoid $\MA$ and its presheaf topos $\Aset \coloneqq \PSh(\MA)$. An object of $\Aset$, which we call a \demph{$\A$-set}, can be regarded as a pair $(Q, \delta)$ consisting of a set $Q$ and a function $\delta \colon Q\times \A \to Q$.
This is the simplest topos of the four topoi in the author's paper \cite{hora2024topoi}. Then, \textit{the $\A$-set of languages} $\Lan \in \ob(\Aset)$ is defined as the set of all languages $\Lan \coloneqq \Pow(\MA)$ equipped with the \demph{left quotient action}
\[
L\ast u \coloneqq \{v\in \MA\mid uv\in L\},
\]
which is also categorically characterized by the universality $\Aset((Q, \delta), \Lan) \cong \Pow(Q)$.

What is the local state classifier of the topos $\Aset$? By the general formula for the local state classifier of a presheaf topos \cite[Example 3.22]{hora2024internal}, we can conclude that $\Xi$ is the $\A$-set of all \demph{right congruences}
\[
\Xi = \left\{\text{equivalence relation }{\sim}\text{ on }\MA\mid \forall u,v,w\in \MA,\;  u\sim v \implies uw \sim vw\right\},
\]
equipped with the right $\A$-action
\[
u\mathrel{({\sim} * w)}v \iff wu \sim wv.
\]
Then, what is $\xi_{\Lan}(L)$? In other words, what is the right congruence categorically associated with a given language $L\in \Lan$?
In general, each component of the colimit cocone $\xi_{(Q, \delta)}\colon (Q, \delta)\to \Xi$ sends a state $q\in Q$ to the right congruence $\xi_{(Q, \delta)}(q) \in \Xi$ defined by
\[
u \mathrel{(\xi_{(Q, \delta)}(q) )} v \iff q\cdot u=q\cdot v.
\]
In particular, the morphism $\xi_{\Lan}\colon \Lan \to \Xi$ sends a language $L\in \Lan$ to
what is called \demph{the Nerode congruence}
\[
u \mathrel{(\xi_{\Lan}(L))}v \iff (L \ast u = L \ast v) \iff \left(\forall w\in \MA, (uw\in L \iff vw\in L)\right),
\]
which is the right congruence used for automata minimization.

In \cite{hora2024topoi}, the hyperconnected geometric morphism $h \colon \Aset \twoheadrightarrow \ofAset$ to the topos of \demph{orbitwise finite $\A$-sets} plays a central role. Here, an orbitwise finite $\A$-set is a $\A$-set $(Q, \delta)$ such that for each element $q\in Q$, its orbit $\{q\cdot w\mid w\in \MA\}\subset Q$ is finite. By construction, the corresponding internal filter $F_{\of} \rightarrowtail \Xi$ is given by
\begin{equation}\label{eq:DefinitionOfFof}
    F_{\of}= \{{\sim} \in \Xi \mid \#(\MA/{\sim})<\infty\}.
\end{equation}
One reason why this particular topos $\ofAset$, which corresponds to the filter $F_{\of}\subset \Xi$, captures the notion of regular languages (see \cite{hora2024topoi}) is the Myhill-Nerode theorem.
\begin{proposition}[Myhill-Nerode theorem]
    A language $L$ is regular if and only if $\xi_{\Lan}(L) \in F_{\of}$.
\end{proposition}
In other words, the pullback diagram (\ref{eq:PullbackDescriptionOfTheCounitAndComonad})
\[
\begin{tikzcd}
        \G_{\of} \Lan\ar[r, "\xi^{F_{\of}}_{\G X}"]\ar[d, "\epsilon_\Lan", tail]\ar[rd, phantom, "\lrcorner", very near start]& F_{\of}\ar[d,tail, "\iota_{F_{\of}}"]\\
        \Lan \ar[r , "\xi_\Lan "']& \Xi.
    \end{tikzcd}
\]
exactly states that $\G_{\of} \Lan$, where $\G_{\of}$ denotes the corresponding lex comonad, is the orbitwise finite $\A$-set of regular languages. In the forthcoming paper, this is the key observation to categorically address the Myhill-Nerode theorem and its generalized versions.

The main theorem of the present paper (Theorem \ref{thm:MainTheorem}) implies that the local state classifier of the topos $\ofAset$ is given by $F_{\of}$.
\begin{corollary}\label{cor:LocalStateClassifierOfOrbitfiniteAset}
    The local state classifier of $\ofAset$ is given by $F_{\of}$ (eq.(\ref{eq:DefinitionOfFof})), equipped with the morphisms $\{\xi_{(Q, \delta)}\colon (Q, \delta)\to F_{\of}\}_{(Q, \delta)\in \ob(\ofAset)}$, where
    \[
    u\mathrel{(\xi_{(Q,\delta)}(q))}v \iff q\cdot u=q\cdot v.
    \]
\end{corollary}

\subsection{Two-sided congruence and Syntactic monoids}
In order to define the \demph{syntactic monoid} of a language, it does not suffice to consider only right congruences. We need \demph{two-sided congruences}
on $\MA$. Recall that an equivalence relation ${\sim} \subset \MA\times \MA$ is said to be a two-sided congruence if it satisfies
\[
\forall u,v,w,w'\in \MA,\; u\sim v \implies  wuw' \sim wvw'.
\]
In algebraic language theory, \demph{the syntactic monoid} $M_L$ of a language $L$ is defined to be the quotient monoid $M_L\coloneqq \MA/{\cong_L}$ of $\MA$, where \demph{the syntactic congruence} $\cong_L$ is the two-sided congruence defined by
\[
u \cong_L v \iff \left (\forall w,w'\in \MA, \; wuw' \in L \iff wvw'\in L\right).
\]
We can rewrite this in terms of the Nerode congruence $\xi_{\Lan}(L)$ as follows:
\begin{align*}
    u \cong_L v 
    &\iff \left (\forall w,w'\in \MA, \; wuw' \in L \iff wvw'\in L\right)\\
    &\iff \left (\forall w,w'\in \MA, \; uw' \in L*w \iff vw'\in L*w\right)\\
    &\iff \forall w\in \MA, \; u \mathrel{(\xi_{\Lan}(L*w))} v\\
    &\iff \forall w\in \MA, \; u \mathrel{(\xi_{\Lan}(L)*w)} v.
\end{align*}
This provides a description of the syntactic congruence in terms of a local state classifier.
\begin{proposition}\label{prop:SyntacticCongruenceIntermsOfLSC}
    For any language $L\in \Lan$, the syntactic congruence $\cong_L$ is given by the infimum\footnote{This is an infinitary meet, which does not necessarily exist in the general setting of elementary topoi. In this particular case of $\Aset$, the local state classifier $\Xi$ has a complete lattice structure.} of the orbit of the Nerode congruence $\xi_{\Lan}(L)$ in $\Xi$:
\[
{\cong_L} = 
\bigwedge_{w\in \MA} \left(\xi_{\Lan}(L)*w\right)
\text{ in } \Xi.
\]
\end{proposition}
What we will use
in the paper \textit{Topoi of automata II} is the following condition for a language $L$:
\begin{equation}\label{eq:NerodeAndMyhillCongruence}
    \text{For any internal filter $F\subset \Xi$, we have }\xi_{\Lan}(L) \in F \iff {\cong_L}\in F.
\end{equation}
In the topos-theoretic framework, this equivalence (\ref{eq:NerodeAndMyhillCongruence}) is crucial for comparing the automata and monoids, as it states that the Nerode congruence $\xi_{\Lan}(L)$, which realizes the minimal automaton, and the syntactic congruence $\cong_L$, which realizes the syntactic monoid, behave in the same way in terms of hyperconnected quotients.

\textbf{This is where we need the normalization operator $\xi_{\Xi}$!}
Since any internal filter $F$ is closed upward, under the $\MA$-action, and taking finite infimums, the condition 
\begin{equation}\label{eq:RegularCongruenceCondition}
     \{\xi_{\Lan}(L)*w\mid w\in \MA\}\subset \Xi \text{ is a finite set}
\end{equation}
is sufficient for ensuring the condition (\ref{eq:NerodeAndMyhillCongruence}).
This condition (\ref{eq:RegularCongruenceCondition}) is equivalent to $\xi_{\Xi}(\xi_{\Lan}(L))\in F_{\of}$, in which the normalization operator $\xi_{\Xi}$ appears! Summarizing what we have observed, we obtain the following proposition.
\begin{proposition}\label{prop:NerodeAndMyhillCongruences}
    For any language $L\in \Lan$ with the property $\xi_{\Xi}(\xi_{\Lan}(L))\in F_{\of}$ and any internal filter $F\subset \Xi$, 
    the following conditions are equivalent:
    \begin{itemize}
        \item The Nerode congruence $\xi_{\Lan}(L)$ belongs to $F$.
        \item The syntactic congruence ${\cong_{L}}$ belongs to $F$.
    \end{itemize}
\end{proposition}

The typical examples of a language that satisfies the condition $\xi_{\Xi}(\xi_{\Lan}(L)) \in F_{\of}$ are regular languages. In its proof, we can use the normalization lemma (Proposition \ref{prop:NormalizationLemma}).
\begin{corollary}\label{cor:RegularlanguageCongruences}
    For any regular language $L\in \Lan$ and any internal filter $F\subset \Xi$, 
    the following conditions are equivalent:
    \begin{itemize}
        \item The Nerode congruence $\xi_{\Lan}(L)$ belongs to $F$.
        \item The syntactic congruence ${\cong_{L}}$ belongs to $F$.
    \end{itemize}
\end{corollary}
\begin{proof}
    In order to apply Proposition \ref{prop:NerodeAndMyhillCongruences}, it suffices to prove $\xi_{\Xi}(\xi_{\Lan}(L))\in F_{\of}$ for every regular language $L$.
    Since a language $L$ is regular if and only if $\xi_{\Lan}(L)\in F_{\of}$, and every internal filter $F$ is upward closed, it suffices to prove $\xi_{\Lan}(L) \leq \xi_{\Xi}(\xi_{\Lan}(L))$. 
    This follows from
    the normalization lemma $\id_{\Xi}\leq \xi_{\Xi}$ (Proposition \ref{prop:NormalizationLemma}).
\end{proof}


\appendix
\section{Preliminaries on elementary topoi}
This appendix summarizes the properties of elementary topoi that are used in this paper.

\begin{lemma}[Jointly epimorphic families in an elementary topos]\label{lem:JointlyEpimorphicFamilyAndSubobject}
    For a (possibly large) family of morphisms $\{f_\lambda \colon X_\lambda \to Y\}_{\lambda \in \Lambda}$ in a category $\E$, we consider the following two conditions:
    \begin{enumerate}
        \item $\{f_\lambda \colon X_\lambda \to Y\}_{\lambda \in \Lambda}$ is jointly epimorphic.
        \item If all morphisms in the family factor through a monomorphism $m\colon S\rightarrowtail Y$, then $m$ is an isomorphism.
    \end{enumerate}
    If $\E$ is balanced, i.e., every monic and epic morphism is an isomorphism, then $(1)$ implies $(2)$. If $\E$ has equalizers, then $(2)$ implies $(1)$. In particular, if $\E$ is an elementary topos, the two conditions $(1),(2)$ are equivalent.
\end{lemma}
\begin{proof}
    First, assuming that the family is jointly epimorphic and $\E$ is balanced, we prove the condition $(2)$. Take an arbitrary monomorphism $m\colon S\rightarrowtail Y$ such that every morphism $f_\lambda$ in the family factors through $m$ as $f_{\lambda} = m\circ f^S_{\lambda}$. For any morphisms $g,h \colon Y \rightrightarrows Z$ such that $g\circ m = h\circ m$, we have $g\circ f_{\lambda} = g\circ m \circ f_{\lambda}^{S} = h\circ m \circ f_{\lambda}^{S} = h\circ f_{\lambda}$ and hence $g=h$. This proves that $m$ is epic. The balancedness assumption implies that $m$ is an isomorphism.

    Next, assuming the condition $(2)$ and that $\E$ has equalizers, we prove $(1)$. Take arbitrary morphisms $g,h \colon Y \rightrightarrows Z$ such that $g\circ f_{\lambda} =h\circ f_{\lambda}$ for any $\lambda$. We prove $g=h$. Let $m\colon S\rightarrowtail Y$ be the equalizer of the two morphisms $g$ and $h$. Then every morphism in the family factors through $m$, and the assumption $(2)$ implies that $m$ is an isomorphism. This proves $g=h$. 
\end{proof}

\begin{lemma}
    \label{lem:sgtNaturality}
    For any morphism $f\colon X\to Y$ in a topos, the diagram
    \[
    \begin{tikzcd}
        X\ar[r,"\sgt_{X}", tail]\ar[d, "f"]&PX\ar[d,"\exists_f"]\\
        Y\ar[r,"\sgt_Y", tail]&PY
    \end{tikzcd}
    \]
    commutes.
\end{lemma}
\begin{proof}
Via the bijection $\E(X, PY) \cong \Sub(X\times Y)$,
    both of the two maps correspond to the subobject
    \[
    \langle\id_X, f \rangle \colon X \rightarrowtail X\times Y.
    \]
    In fact, the upper right part corresponds to the image of the composite
    \[
    \begin{tikzcd}
        X \ar[r,rightarrowtail, "\Delta_X"] &X\times X \ar[r,"\id_X \times f"]& X\times Y,
    \end{tikzcd}
    \]
    and the lower left part corresponds to the pullback 
    \[
    \begin{tikzcd}
        X\ar[r, "f"] \ar[d, rightarrowtail, "{\langle \id_X, f\rangle }"'
        ] \ar[rd, phantom, "\lrcorner", very near start]
        & Y\ar[d, "\Delta_Y", rightarrowtail]\\
        X\times Y \ar[r,"f\times \id_Y"] & Y\times Y.
    \end{tikzcd}
    \]
\end{proof}

\begin{lemma}
\label{lem:sgtExtNaturality}
    For any monomorphism $m \colon X\rightarrowtail Y$ in a topos, the diagram 
    \[
    \begin{tikzcd}
        X\ar[r,"\sgt_{X}", tail]\ar[d,tail, "m"]&PX\\
        Y\ar[r,"\sgt_{Y}", tail]&PY\ar[u,"m^{-1}"']
    \end{tikzcd}
    \]
    commutes.
\end{lemma}
\begin{proof}
    The diagram
    \[
    \begin{tikzcd}[row sep = 10 pt]
        X\times X\ar[rd,"\delta_X"]\ar[dd,"m\times m"', tail]&\\
        & \Omega\\
        Y \times Y\ar[ru, "\delta_Y"']
    \end{tikzcd}
    \]
    commutes since the morphism $m$ is monic.
    Taking the transpose using the naturality of the three-variable adjunction $\E(A\times B , C)\cong\E(A, C^B)$, we obtain the commutativity as stated.
\end{proof}

\printbibliography
\end{document}